\documentclass{amsart}%
\usepackage{amsmath}
\usepackage{amssymb}
\usepackage{amsfonts}
\usepackage{graphicx}%
\providecommand{\U}[1]{\protect\rule{.1in}{.1in}}
\newtheorem{theorem}{Theorem}
\theoremstyle{plain}

\newtheorem{lemma}{Lemma}

\begin{document}
\title[A new proof of a theorem of Tietze]{ A new proof of a theorem of Tietze}
\author{Daniel Duverney}
\address{110, rue du chevalier fran\c{c}ais, 59000 Lille, France}
\email{daniel.duverney@orange.fr}
\author{Iekata Shiokawa}
\address{13-43, Fujizuka-cho, Hodogaya-ku, Yokohama 240-0031, Japan}
\email{shiokawa@beige.ocn.ne.jp}
\date{March 5, 2022}

\begin{abstract}
We give a new proof of Tietze Theorem on the convergence of infinite
semi-regular continued fractions.

\end{abstract}
\maketitle

The purpose of this note is to give a new proof of Theorem \ref{TietzeTh}
below, known as \textit{Tietze Theorem \cite{Per}}. By Theorem \ref{TietzeTh},
any infinite semi-regular continued fraction (\cite{Tietze},\cite{Per}%
,\cite{DS}) is convergent.

\begin{theorem}
\label{TietzeTh}Assume that the infinite continued fraction
\begin{equation}
\alpha:=b_{0}+\frac{a_{1}}{b_{1}}%
\genfrac{}{}{0pt}{}{{}}{+}%
\frac{a_{2}}{b_{2}}%
\genfrac{}{}{0pt}{}{{}}{+\cdots+}%
\frac{a_{n}}{b_{n}}%
\genfrac{}{}{0pt}{}{{}}{+\cdots}
\label{Int1}%
\end{equation}
satisfies the following conditions:%
\begin{equation}
a_{n}\in\left\{  -1,1\right\}  ,\qquad b_{n}\in\left[  1,+\infty\right[
,\qquad b_{n}+a_{n+1}\geq1\qquad\left(  n\geq1\right)  . \label{Int2}%
\end{equation}
Then $\alpha$ is convergent.
\end{theorem}

Tietze original proof of Theorem \ref{TietzeTh} (1911) rested on a geometrical
representation of the convergents $p_{n}/q_{n}$ as points $(p_{n},q_{n})$ in
the plane \cite{Tietze}. Shortly after (1913), Oskar Perron gave a different
proof by using continuants \cite{Per}. Almost 100 years later, Keith Matthews
challenged Alan Offer to find a proof of Theorem \ref{TietzeTh} avoiding the
use of continuants. In \cite{Offer}, Offer succeeded in doing so in the proof
of Lemma \ref{LemInf} below, but not in the rest of the proof. In this note we
develop Offer's ideas and give a new simple proof of Tietze Theorem without
using at all continuants.\medskip

Let $\alpha$ be the continued fraction defined by (\ref{Int1}). As usual, we
define for $n\geq1$%
\begin{equation}
\left\{
\begin{array}
[c]{lll}%
p_{-1}=1, & p_{0}=b_{0}, & p_{n}=b_{n}p_{n-1}+a_{n}p_{n-2},\\
q_{-1}=0, & q_{0}=1, & q_{n}=b_{n}q_{n-1}+a_{n}q_{n-2}.
\end{array}
\right.  \label{Int5}%
\end{equation}
It is well known that%
\begin{equation}
b_{0}+\frac{a_{1}}{b_{1}}%
\genfrac{}{}{0pt}{}{{}}{+}%
\frac{a_{2}}{b_{2}}%
\genfrac{}{}{0pt}{}{{}}{+\cdots+}%
\frac{a_{n}}{b_{n}}=\frac{p_{n}}{q_{n}}\qquad\left(  n\geq1\right)  ,
\label{Int6}%
\end{equation}
and an easy induction using (\ref{Int5}) shows that%
\begin{equation}
p_{n}q_{n-1}-p_{n-1}q_{n}=\left(  -1\right)  ^{n-1}a_{1}a_{2}\cdots
a_{n}\qquad\left(  n\geq1\right)  , \label{Int7}%
\end{equation}
which yields immediately%
\begin{equation}
\frac{p_{n}}{q_{n}}=b_{0}+\sum_{k=1}^{n}\frac{\left(  -1\right)  ^{k-1}%
a_{1}a_{2}\cdots a_{k}}{q_{k-1}q_{k}}. \label{Int8}%
\end{equation}

First we prove Theorem \ref{TietzeTh} in the following very special case,
connected to negative continued fractions \cite{DS}.

\begin{lemma}
\label{LemNCF}Assume that $a_{n}=-1$ for all $n\geq1.$ Then the continued
fraction (\ref{Int1}) is well defined and convergent.
\end{lemma}

\begin{proof}
We follow \cite{Cahen}. By (\ref{Int5}) and (\ref{Int2}) we have for $n\geq0$%
\[
q_{n+1}-q_{n}=\left(  b_{n+1}-1\right)  q_{n}-q_{n-1}\geq q_{n}-q_{n-1}%
\geq\ldots\geq q_{0}-q_{-1}=1.
\]
Hence $q_{n}\geq n+1$ for all $n\geq0,$ which proves that the continued
fraction (\ref{Int1}) is well defined and convergent by (\ref{Int8}).
\end{proof}

We now return to the general case. Since $q_{n+1}<q_{n}$ when $\left(
a_{n+1},b_{n+1}\right)  =\left(  -1,1\right)  ,$ the above proof cannot be used.

\begin{lemma}
\label{Lem1}Let $q_{n}$ be defined by (\ref{Int5}). We have for $n\geq0$%
\begin{equation}
q_{n}\geq1,\quad q_{n}+a_{n+1}q_{n-1}\geq1. \label{Offer}%
\end{equation}

\end{lemma}

\begin{proof}
By following Offer \cite[Lemma 1]{Offer}, we first observe that
\begin{equation}
q_{n+1}+a_{n+2}q_{n}=\left(  b_{n+1}+a_{n+2}\right)  q_{n}+a_{n+1}q_{n-1}\geq
q_{n}+a_{n+1}q_{n-1}\quad\left(  n\geq0\right)  . \label{A1}%
\end{equation}
Hence by induction $q_{n}+a_{n+1}q_{n-1}\geq q_{0}+a_{1}q_{-1}=1$ for all
$n\geq0,$ which yields
\begin{equation}
q_{n+1}=b_{n+1}q_{n}+a_{n+1}q_{n-1}\geq q_{n}+a_{n+1}q_{n-1}\geq1\quad\left(
n\geq0\right)  . \label{A2}%
\end{equation}
Therefore (\ref{Offer}) holds.
\end{proof}

Lemma \ref{Lem1} shows that the convergents of the continued fraction $\alpha$
defined by (\ref{Int1}) and (\ref{Int2}) are well defined since $q_{n}\geq1$
for all $n\geq0.\smallskip$

The following lemma is proved in \cite{Per} by using continuants and in
\cite{Offer} without using them. We give here a new proof, also without using continuants.

\begin{lemma}
\label{LemInf}Let $q_{n}$ be defined by (\ref{Int5}). Then $\lim
_{n\rightarrow\infty}q_{n}=\infty.$
\end{lemma}

\begin{proof}
As observed in \cite[Formula (3)]{Offer}, it is clear by (\ref{A1}) and
(\ref{A2}) that
\begin{equation}
q_{n}\geq q_{m}+a_{m+1}q_{m-1}\geq1\qquad\left(  n>m\geq1\right)  . \label{A3}%
\end{equation}
Assume that there exists $A\geq1$ such that $1\leq q_{n}\leq A$ for infinitely
many $n.$ Then there exists $B\geq1$ and an increasing sequence $n_{k}$ such
that $\lim_{k\rightarrow\infty}q_{n_{k}}=B.$ Let $C\geq1$ be the least real
number having this property. By (\ref{A3}) with $n=n_{k+1}$ and $m=n_{k}$, we
have for all large $k$%
\[
C+\frac{1}{3}\geq C-\frac{1}{3}+a_{n_{k}+1}q_{n_{k}-1}\quad\text{and}\quad
C+\frac{1}{3}+a_{n_{k}+1}q_{n_{k}-1}\geq1.
\]
This yields immediately $a_{n_{k}+1}=-1,$ and therefore $q_{n_{k}-1}\leq
C-2/3$ for all large $k.$ This contradicts the minimality of $C$ and proves
Lemma \ref{LemInf}.
\end{proof}

\begin{lemma}
\label{Lem-x-n,k}Define for $n\geq0$ and $k\geq1$%
\begin{equation}
x_{n,k}:=\frac{a_{n+1}}{b_{n+1}}%
\genfrac{}{}{0pt}{}{{}}{+}%
\frac{a_{n+2}}{b_{n+2}}%
\genfrac{}{}{0pt}{}{{}}{+\cdots+}%
\frac{a_{n+k}}{b_{n+k}}. \label{x-n-k-0}%
\end{equation}
Then for $n\geq0$ and $k\geq1$%
\begin{equation}
0<x_{n,k}\leq1\text{ if }a_{n+1}=1,\qquad-1\leq x_{n,k}<0\text{ if }%
a_{n+1}=-1. \label{x-n-k}%
\end{equation}

\end{lemma}

\begin{proof}
The proof follows \cite[Proposition 3]{DS} and is by induction on $k.$ If
$k=1,$ then $x_{n,1}=a_{n+1}/b_{n+1}$ and so (\ref{Int2}) implies
(\ref{x-n-k}) for $k=1$ and any $n\geq0.$ Assume that (\ref{x-n-k}) holds for
some $k\geq1$ and any $n\geq0.$ Then by (\ref{x-n-k-0}) we have%
\begin{equation}
x_{n,k+1}=\frac{a_{n+1}}{b_{n+1}+x_{n+1,k}},\label{x-n-k-1}%
\end{equation}
where for any $n\geq0$%
\begin{equation}
0<x_{n+1,k}\leq1\text{ if }a_{n+2}=1,\qquad-1\leq x_{n+1,k}<0\text{ if
}a_{n+2}=-1.\label{x-n-k-2}%
\end{equation}
If $b_{n+1}\geq2,$ then by (\ref{x-n-k-2})%
\begin{equation}
b_{n+1}+x_{n+1,k}\geq1.\label{x-n-k-3}%
\end{equation}
Otherwise, we have $b_{n+1}<2$ and therefore $a_{n+2}=1$ by (\ref{Int2}), and
so we find again (\ref{x-n-k-3}) by the first inequality in (\ref{x-n-k-2}).
Hence (\ref{x-n-k-1}) and (\ref{x-n-k-3}) imply%
\[
0<x_{n,k+1}\leq1\text{ if }a_{n+1}=1,\qquad-1\leq x_{n,k+1}<0\text{ if
}a_{n+1}=-1,
\]
which proves (\ref{x-n-k}) by induction.
\end{proof}

\begin{lemma}
\label{LemYes}Let $x_{n,k}$ be defined by (\ref{x-n-k-0}). Then%
\begin{equation}
\frac{p_{n+k}}{q_{n+k}}=\frac{p_{n}+x_{n,k}p_{n-1}}{q_{n}+x_{n,k}q_{n-1}%
}\qquad\left(  n\geq0,k\geq1\right)  . \label{FormYes}%
\end{equation}

\end{lemma}

\begin{proof}
By induction on $n.$ Clearly (\ref{FormYes}) is true for $n=0$ and every
$k\geq1$ since%
\[
\frac{p_{0}+x_{0,k}p_{-1}}{q_{0}+x_{0,k}q_{-1}}=b_{0}+\frac{a_{1}}{b_{1}}%
\genfrac{}{}{0pt}{}{{}}{+}%
\frac{a_{2}}{b_{2}}%
\genfrac{}{}{0pt}{}{{}}{+\cdots+}%
\frac{a_{k}}{b_{k}}=\frac{p_{k}}{q_{k}}.
\]
Assume that (\ref{FormYes}) is true for a given $n\geq0$ and all $k\geq1.$ By
(\ref{x-n-k-1}), we see that%
\[
\frac{p_{n+1+k}}{q_{n+1+k}}=\frac{p_{n+k+1}}{q_{n+k+1}}=\frac{p_{n}%
+x_{n,k+1}p_{n-1}}{q_{n}+x_{n,k+1}q_{n-1}}=\frac{p_{n}\left(  b_{n+1}%
+x_{n+1,k}\right)  +a_{n+1}p_{n-1}}{q_{n}\left(  b_{n+1}+x_{n+1,k}\right)
+a_{n+1}q_{n-1}},
\]
which proves (\ref{FormYes}) by induction.\smallskip
\end{proof}

From Lemmas \ref{LemYes} and \ref{Lem-x-n,k} we deduce immediately that%
\begin{equation}
\left\vert \frac{p_{n+k}}{q_{n+k}}-\frac{p_{n}}{q_{n}}\right\vert \leq\frac
{1}{q_{n}\left\vert q_{n}+x_{n,k}q_{n-1}\right\vert }\qquad\left(
n\geq0,k\geq1\right)  . \label{FormYoup}%
\end{equation}
Now we prove Tietze Theorem. By Lemma \ref{LemNCF}, we can assume that there
exist infinitely many $m$ such that $a_{m+1}=1.$ For every large $n,$ let
$m=m(n)$ be the greatest integer less than $n$ such that $a_{m+1}=1.$ Clearly
$m$ tends to infinity with $n.$ For every large $n$ and every $k\geq1,$ we
have by (\ref{FormYoup})%
\begin{align*}
\left\vert \frac{p_{n+k}}{q_{n+k}}-\frac{p_{n}}{q_{n}}\right\vert  &
\leq\left\vert \frac{p_{n+k}}{q_{n+k}}-\frac{p_{m}}{q_{m}}\right\vert
+\left\vert \frac{p_{n}}{q_{n}}-\frac{p_{m}}{q_{m}}\right\vert \\
&  \leq\frac{1}{q_{m}\left\vert q_{m}+x_{m,n+k-m}q_{m-1}\right\vert }+\frac
{1}{q_{m}\left\vert q_{m}+x_{m,n-m}q_{m-1}\right\vert }.
\end{align*}
However $x_{m,n+k-m}>0$ and $x_{m,n-m}>0$ by Lemma \ref{Lem-x-n,k}. Hence%
\[
\left\vert \frac{p_{n+k}}{q_{n+k}}-\frac{p_{n}}{q_{n}}\right\vert \leq\frac
{2}{q_{m\left(  n\right)  }^{2}}%
\]
for all large $n$ and all $k\geq1.$ As $\lim_{n\rightarrow\infty}q_{m\left(
n\right)  }=\infty$ by Lemma \ref{LemInf}, $p_{n}/q_{n}$ is a Cauchy sequence,
which proves Tietze Theorem.


\begin{thebibliography}{9}                                                                                                %


\bibitem {Cahen}E. Cahen, \textit{Th\'{e}orie des nombres}, Tome Second,
Hermann, 1924.

\bibitem {DS}D. Duverney and I. Shiokawa, \textit{Irrationality exponents of
semi-regular continued fractions},

ArXiv:2202.10738v1 (22 Feb 2022).

\bibitem {Offer}A. Offer, \textit{Continuants and semi-regular continued
fractions} (2008),\qquad\qquad\ 

http://www.numbertheory.org/PDFS/continuant.pdf.

\bibitem {Per}O. Perron, \textit{Die Lehre von den Kettenbr\"{u}chen},
Teubner, 1913.

\bibitem {Tietze}H. Tietze, \textit{\"{U}ber Kriterien f\"{u}r Konvergenz und
Irrationalit\"{a}t unendlichen Kettenbr\"{u}che},

Math. Ann. 70 (1911), 236-265.
\end{thebibliography}
\end{document}